\documentclass[11pt,a4paper]{amsart}
\usepackage[toc,page]{appendix}
\usepackage{a4}
\usepackage[active]{srcltx}
\usepackage{amsmath,bbm,esint}
\usepackage{amssymb}
\usepackage{amsfonts}
\usepackage{mathrsfs} 
\usepackage{graphicx}
\usepackage{color}
\usepackage[colorlinks=true,urlcolor=blue,
citecolor=red,linkcolor=blue,linktocpage,pdfpagelabels,bookmarksnumbered,bookmarksopen]{hyperref}

\def\Tau{\Upsilon}

\def\d{\delta}

\def\l{\lambda}
\def\vphi{\varphi}

\def\n{\nu}

\def\t{\tau}
\def\e{\varepsilon}

\def\Om{\Omega}

\newcommand{\cK}{{\mathcal K}}

\newcommand{\R}{{\mathbb R}}
\newcommand{\W}{{\mathbb W}}

\newcommand{\mres}{\mathbin{\vrule height 1.6ex depth 0pt width
0.13ex\vrule height 0.13ex depth 0pt width 1.3ex}}

\newcommand{\sL}{\mathscr{L}}
\newcommand{\sH}{{\mathscr H}}
\newcommand{\sM}{{\mathscr M}}
\newcommand{\sP}{{\mathscr P}}

\newcommand{\ds}{\displaystyle}

\newcommand{\diver}{\nabla\cdot}

\newcommand{\iin}{{\rm{in}}}
\newcommand{\oon}{{\rm{on}}}

\renewcommand{\ae}{{\rm{a.e.}}}
\newcommand{\dd}{\,{\rm d}}
\newcommand{\id}{{\rm id}}

\renewcommand{\u}{{\textbf{u}}}
\renewcommand{\v}{{\textbf{v}}}

\renewcommand{\n}{{\hat n}}

\newcommand{\adm}{{\rm adm}}

\newtheorem{remark}{\textbf{Remark}}[section]
\newtheorem{theorem}{\textbf{Theorem}}[section]
\newtheorem{lemma}[theorem]{\textbf{Lemma}}
\newtheorem{corollary}[theorem]{\textbf{Corollary}}

\newtheorem{definition}[remark]{\textbf{Definition}}

\numberwithin{equation}{section}

\title[]{Uniqueness issues for evolution equations with density constraints}  

\author[S. Di Marino]{Simone Di Marino} 
\address{Laboratoire de Math\'ematiques d'Orsay, Universit\'e Paris-Sud, CNRS, Universit\'e Paris-Saclay, France}
\email{simone.dimarino@math.u-psud.fr}

\author[A.R. M\'esz\'aros]{Alp\'ar Rich\'ard M\'esz\'aros}  
\date{\today}
\address{Department of Mathematics, UCLA, Los Angeles, California, USA}
\email{alpar@math.ucla.edu} 

\begin{document}
\maketitle
\begin{abstract}
In this paper we present some basic uniqueness results for evolutive equations under density constraints. First, we develop a rigorous proof of a well-known result (among specialists) in the case where the spontaneous velocity field satisfies a monotonicity assumption: we prove the uniqueness of a solution for first order systems modeling crowd motion with hard congestion effects, introduced recently by \emph{Maury et al.} The monotonicity of the velocity field implies that the $2-$Wasserstein distance along two solutions is $\lambda$-contractive, which in particular implies uniqueness. In the case of diffusive models, we prove the uniqueness of a solution passing through the dual equation, where we use some well-known parabolic estimates to conclude an $L^1-$contraction property. In this case, by the regularization effect of the non-degenerate diffusion, the result follows even if the given velocity field is only $L^\infty$ as in the standard Fokker-Planck equation.
\end{abstract}

\section{Introduction and preliminaries}

Recently, modeling crowd behavior has received a lot of attention in applied mathematics. These models actually are in the heart of many other ones coming from biology (for instance cell migration, tumor growth, pattern formations in animal populations, etc.), particle physics and economics (see for example the recently introduced models of Mean Field Games, \cite{lasry1,lasry2,lasry3}). For more details on these models we direct the reader to the non-exhaustive list of works \cite{Cha1, Col, Cos, CriPicTos, Helb1, Helb3, Hend, Hug1, Hug2}.
In all these models the question of {\it congestion} can play a crucial role. Indeed, from the modeling point of view one could have some `singularities' if individuals want to occupy the same spot. In this paper, we will work with equations which model some type of  congestion effects in crowd motion models (for a more detailed description of these models we direct the reader to the references \cite{MauRouSan1, MauRouSanVen, MesSan}). These systems read as 
$$
\begin{cases}
\partial_t\rho_t-\nu\Delta\rho_t+\nabla\cdot(\rho_t\v_t)=0\\
\rho_t\le1,\ \rho|_{t=0}=\rho_0\\
\v_t=P_{\adm(\rho_t)}[\u_t].
\end{cases}
$$
In the above system $\rho_t$ represents the density of a crowd (at time $t$) that moves in $\Omega\subset\R^d$ for a time $T>0$ accordingly to the prescribed velocity $\u_t$, a given field that everyone would follow in the absence of the others. $\nu\in\{0,1\}$ is just a parameter: if $\nu=0$, one deals with first order systems (the density of the population is just transported by some vector field) introduced in \cite{MauRouSan1}, while for $\nu=1$ one has a second order system (in addition to the transport, the population is also affected by some randomness, modeled by a non-degenerate diffusion) studied in \cite{MesSan}. In order to preserve a density constraint (we suppose that $\rho_t$ does not exceed a given threshold, let us say $1$), at each moment one modifies $\u_t$ to $\v_t,$ a field that is the closest to $\u_t$ (in the $L^2$ sense) and it is an admissible velocity field, i.e. the set $\adm(\rho_t)$ represents the fields that do not increase the density on already saturated zones $\{\rho_t=1\}.$ 

In generic congested models a pressure is also acting (Darcy's law): this pressure is preventing congestion and, according to various models, it is an increasing function of the density. In the usual \textit{soft} congestion models we have some power law of the type $p=\rho^m$ (porous medium equation). However this choice cannot prevent $\rho$ to be very high, while it is clear that a crowd of people cannot have a density higher than a certain threshold (studies say that the maximum density is $4.5$ people$/m^2$), which we put for convenience equal to $1$: we will refer to the constraint $\rho \leq 1$ as an \textit{hard} congestion effect. We shall present later how the pressure field appears in the models considered in this paper. 

Introducing general multivalued monotone operators $\beta$, the above systems can be written in a compact form as
\begin{equation}\label{eqn:main}
\begin{cases} \partial_t \rho_t + \nabla \cdot ( \u_t \rho_t ) = \Delta p_t  \quad & \text{ in }(0,T)\times\Omega\\  (\u_t \rho_t - \nabla p) \cdot \hat{\mathbf{n}}  = 0 & \text{ on } (0,T)\times\partial \Omega \\  p_t \in \beta (\rho_t). 
\end{cases}
\end{equation}
The two cases $\nu=0$ and $\nu=1$ correspond to two special operators $\beta_1$ and $\beta_2$, namely

$$ \beta_1 (\rho ) = \begin{cases} 0 \quad & \text{ if }\rho <1 \\ [0,+\infty] & \text{ if }\rho=1 \\ +\infty & \text{ if } \rho>1, \end{cases} \qquad \quad 
\beta_2 (\rho ) = \begin{cases} \rho \quad & \text{ if }\rho <1 \\ [1,+\infty] & \text{ if }\rho=1 \\ +\infty & \text{ if } \rho>1; \end{cases} $$
 notice in particular that $\beta_2(\rho) = \beta_1(\rho)+\rho$. The \textit{hard} congestion effect is due to the fact that $\beta_{i}(\rho)= +\infty$ whenever $\rho >1$: in fact, this will force $\rho$ to be always less than $1$. 
It is worthwhile noticing  that this problem has some features in common with the Hele-Shaw model and the Stefan problem (see for instance \cite{Igb2, IgbShiWit}), namely in both problems there is a degenerate monotone operator linking  the density and the diffusive part. However there are also big differences: in our case we treat also the convection term, which can depend not smoothly on the position, while usually in the Hele-Shaw models there is a source term for the mass, which we do not treat here since we want to model a crowd moving inside a domain $\Omega$. This modeling assumption is also the reason why we consider the no-flux Neumann boundary condition, since we want neither people exiting nor entering the domain.

Moreover our equations can be seen also as a quasilinear elliptic-parabolic system with very degenerate nonlinearity (see \cite{Otto,Igb, Carrillo}), for which the issue of uniqueness is still an open problem when we add a driving vector field without imposing an entropy condition. In this context we may write our equations using the variable $u=p+\rho$ as
$$ \partial_t g( u )  + \nabla \cdot ( \Phi (t,x, u ))= \Delta b_i( u ), $$
where $b_1(u)=(u-1)_+$, $g(u)=u-b_1(u)$, $b_2(u)=u$ and $\Phi(t,x,u ) = \u_t(x) g(u)$. In the case $i=1$ we have a double degeneracy since both $b_1$ and $g$ have a flat part, instead for the case $i=2$ we have only a degenerate part in $g(u)$. However again the (possibly) rough coefficient in $\Phi$ rules out results already present in the literature.  For problems of this form we expect an $L^1-$contraction result, see for example \cite{Carrillo}; notice also that we do not need a concept of entropic solution since the equation reduces to a linear one in the joint variable $(\rho,p)$ (see the systems \eqref{main2} below). Thus the usual concept of weak solutions can be considered for our purposes.

In fact, we will be dealing always with the following reformulations of the systems, using the fact that $p \in \beta(\rho)$ if and only if $\rho \leq 1$ and $p \in H^1_{\rho}(\Omega)$ (see its definition in \eqref{def:h1r}):
\begin{equation}\label{main2} \begin{cases} \partial_t \rho_t + \nabla \cdot ( \u_t \rho_t ) = \Delta p_t  \; & \text{ in }(0,T)\times\Omega\\  (\u_t \rho_t - \nabla p) \cdot \hat{\mathbf{n}}  = 0 & \text{ on } (0,T)\times\partial \Omega \\  \rho_t \leq 1 , \, p_t \in H_{\rho_t}^1(\Omega).  \end{cases} \
 \begin{cases}  \partial_t \rho_t + \nabla \cdot ( \u_t \rho_t ) = \Delta (p_t  + \rho_t ) \quad & \text{ in }(0,T)\times\Omega\\  (\u_t \rho_t - \nabla p - \nabla \rho_t) \cdot \hat{\mathbf{n}}  = 0 & \text{ on } (0,T)\times\partial \Omega \\  \rho_t \leq 1 , \, p_t \in H_{\rho_t}^1(\Omega).  \end{cases} 
\end{equation}

In Subsection \ref{oneone} we will derive in another way these systems: that derivation justifies also the regularity assumption of $p_t$.

A very powerful tool to attack these type of macroscopic hard-congestion problems -- where we impose a density constraint on the density of the population -- is the theory of optimal transport (see \cite{villani, OTAM}), as we can see in the recent works \cite{MauRouSan1, MauRouSan2, MauRouSanVen, MesSan, AleKimYao}. In this framework, the density of the agents satisfies a continuity, or a Fokker-Planck equation (with a velocity field taking into account the congestion effects) and can be seen as a curve in the Wasserstein space.

Our aim in this paper is to prove some basic results of uniqueness of the solutions in this setting. As far as we are aware of, this question is a missing puzzle in full generality in the models studied in \cite{MauRouSan1, MauRouSan2, MauRouSanVen, MesSan}. Let us remark that the uniqueness question is crucial if one wants to include this type of models into a larger system and one aims to show existence results by fixed point methods, as it is done for Mean Field Games in general for instance (for example as in \cite{Por}). 

We will treat two different cases, with very different approaches: in the first one we simply consider a crowd driven by a given velocity field and subject to a density constraint. In this case the assumption that the velocity field is monotone will be crucial in order to prove a $\lambda$-contraction result for the solutions, that will imply uniqueness. In the second case we add a diffusive term, which models some randomness in the crowd movement (see \cite{MesSan} for recent developments and existence results in this setting); in this case we prove a standard $L^1-$contraction property passing to the dual problem and proving there existence for sufficiently generic data. In this case a major role is played by the regularizing effect of the Laplacian, which allows us to prove the result even if the velocity field is merely bounded.


Let us underline that the core of both methods is classical in the literature.

\begin{itemize}
\item[(i)] On one hand we use the differentiation of the squared Wasserstein distance along two solutions of continuity-type equations, and then use a Gr\"onwall-type argument to show a contractive property and hence uniqueness of solutions for these evolution equations. This type of proof is very common in the optimal transport theory; nevertheless our analysis requires some finer nontrivial new ideas, mainly because of the appearance of the new pressure variable which we can formally identify as a term associated to the subdifferential of the indicatrix function of the (geodesically convex) set $\{ \rho \in \sP (\Om) \; : \; \rho \leq 1  \}$. Moreover, we believe that by the nature of our problem the machinery of the $L^1-$contraction through the doubling and re-doubling of the variables (successfully used for instance in \cite{Otto} and \cite{Carrillo}) seems to be very difficult, complicated and heavy to adapt to our situation, while the method proposed by us is simple and elegant. In addition, in this context we give new proofs for some of the used results from the theory of optimal transport.  

\item[(ii)] On the other hand, in the second order case (when we add a non-degenerate diffusion term into the model) the idea to pass to the dual equation to show the uniqueness is very similar to the techniques developed already in \cite{Cro} for instance; the method to obtain the $L^1-$contraction follows the same lines of \cite{Bou}.  However the  PDE  studied here seems to be new in this context and the result of uniqueness is per se interesting in the theory of crowd motion. 
\end{itemize}

We underline moreover that the two methods used in the two models are mutually exclusive: the $W_2$ distance along two solutions of the second order model would be contractive if the vector field would be monotone, but we do not require this assumption; the parabolic estimates used for the second model highly rely on the fact that one has a non-degenerate diffusion in the system, which is clearly not the case for the first order model.

We remark also that we expect a $L^1-$contraction result also in the first case, since it can be seen as a doubly degenerate quasilinear elliptic-parabolic equation \cite{Otto}. In Section \ref{ss:l1} we provide a sketch of this result in the case in which the velocity field is monotone: we underline that we use the uniqueness proved in Section \ref{two}.

\subsection{The density constraint: admissible velocities and pressures}\label{oneone}

In order to model crowd movement in the macroscopic setting with hard congestion, we work in a convex bounded domain $\Om\subset\R^d$  such that $|\Om|>1$. The evolution of the crowd will be analyzed by the evolution of its density, which is assumed to be a probability measure on $\Om$. The condition we want to impose is a bound on the density of the crowd, which is considered to be always less than or equal to $1$. In particular the set of admissible measures will be denoted by 
$$\cK_1:=\{\rho\in\sP(\Omega):\rho\le 1\ {\rm{a.e.}}\}$$ 
(here and after we identify $\rho$ with its density with respect to the Lebesgue measure $\mathscr{L}^d$).

As for the velocities, we need to impose that the density is not increasing when it is saturated: informally we would say that $\v$ is an admissible velocity for the measure $\rho \in\cK_1$ if $ \diver \v \geq 0$ in the set $\{ \rho=1\}$ and $\v \cdot \hat{n} \leq 0$ on $\partial \Omega$ where $\hat{n}$ is the outward normal. In order to make a rigorous definition we have to introduce the set of pressures:
\begin{equation}\label{def:h1r} 
H^1_{\rho}(\Omega) := \{ p \in H^{1}(\Omega) \; : \; p \geq 0 , \;\, p (1-\rho) = 0\;{\rm{a.e.}}\}.
\end{equation}

Then, using the integration by parts formula, formally for $p\in H^1_\rho(\Omega)$
and $\v$ admissible one should have
$$0\leq \int_{\Omega} p \diver \v  \dd x - \int_{\partial \Omega} p \v \cdot \hat{n} \dd\sH^{d-1}= - \int_{\Omega} \v \cdot \nabla p  \dd x$$ and so we can define
$$ \adm (\rho) :=\left\{ \v \in L^2(\Omega; \R^d ) \; : \; \int_{\Omega} \v \cdot \nabla p \dd x \leq 0 \, \text{ for every } p \in H^1_{\rho}(\Omega) \right\}. $$
In the sequel we will denote by $P_{\adm(\rho)}$ the $L^2(\Omega, \mathscr{L}^d; \R^d)-$projection onto the cone $\adm(\rho)$.

 Now, in order to preserve the constraint $\rho \leq 1$, we impose that the velocity always belongs to $\adm (\rho )$ and so a generic evolution equation with density constraint will be
$$ \begin{cases} \partial_t \rho_t + \diver (\rho_t \v_t) = 0 \\ \rho_t \leq 1, \, \v_t \in \adm (\rho_t). \end{cases} $$
One of the simplest such model is when we have a prescribed time-dependent velocity field $\u_t$ and we want to impose that the velocity $\v_t$ is the nearest possible to $\u_t$, time by time. This describes a situation where the crowd wants to have the velocity $\u_t$ but it cannot, because of the density constraint, and so it adapts its velocity trying to deviate as little as possible: this will result in an highly nonlocal and discontinuous effect. The first order problem hence reads as
\begin{equation}\label{fp0}
\begin{cases}
\partial_t\rho_t+\nabla\cdot(\rho_t\v_t)=0 \quad \text{ in }(0,T)\times\Omega\\
\rho_t\le1,\ \rho|_{t=0}=\rho_0\\
\v_t=P_{\adm(\rho_t)}[\u_t],
\end{cases}
\end{equation}
where the first equation is meant in the weak sense and the minimal hypothesis in order to have a well defined projection is $\u \in L^2([0,T] \times \Om)$. In the following lemma we characterize the projection of the velocity field:

\begin{lemma}\label{lem:projection} 
Let $\rho \in \sP (\Omega)$ such that $\rho \leq 1$ a.e. and let $\u \in L^2(\Omega; \R^d)$. Then there exists $p \in H^1_\rho(\Om)$ such that $ P_{\adm (\rho)} [ \u ] = \u - \nabla p$. Furthermore $p$ is characterized by
\begin{itemize}
\item[(i)] $\ds \int_{\Omega} \nabla p  \cdot ( \u - \nabla p ) \dd x =0$;
\item[(ii)] $\ds \int_{\Omega} \nabla q \cdot (\u - \nabla p) \dd x \leq 0$, for all $q \in H^1_\rho(\Om)$.
\end{itemize}
\end{lemma}

\begin{proof} Let us set $K=\{ \nabla p \; : \; p \in H^1_\rho(\Om) \}$. It is easy to see that $K$ is a closed cone in $L^2(\Om;\R^d)$.
Let us recall that the polar cone to $K$ is defined as 
$$K^o:=\left\{\v\in L^2(\Om;\R^d)\; : \; \int_\Om \v\cdot\nabla q\dd x\le 0,\ \forall\ \nabla q\in K \right\}.$$
By the definition of the admissible velocities we have $\adm (\rho ) = K^o$. Moreau decomposition applied to $K$ and $K^o$ gives
$$ \u  = P_K [\u] + P_{K^o} [ \u ]  \qquad \forall \u \in L^2(\Omega; \R^d).$$
This proves the claims (i) and (ii), since these are precisely the conditions of being the projection onto a cone. 
\end{proof}
\begin{corollary}\label{cor:proj}
Lemma \ref{lem:projection} ${\rm{(i)}}$ implies that $\ds \int_{\Om} | \u|^2\dd x = \int_{\Om} | \nabla p|^2\dd x + \int_{\Om} | \u - \nabla p|^2\dd x$, and so in particular we get
$$\ds\int_\Om|\nabla p|^2\dd x\le\int_\Om|\u|^2\dd x, \qquad \quad \int_\Om|P_{\adm (\rho)} [ \u ]|^2\dd x\le\int_\Om|\u|^2\dd x.$$
\end{corollary}

Using Lemma \ref{lem:projection}  we get that if $(\rho, \v)$ is a solution to  \eqref{fp1} then $\v_t = \u_t - \nabla p_t$. Now, using that $\rho_t \nabla p_t = \nabla p_t$, we have that $(\rho, p)$ is a solution to 
\begin{equation}\label{fp1}
\begin{cases}
\partial_t\rho_t+\nabla\cdot(\rho_t\u_t)=\Delta p_t, \quad & \text{in }(0,T)\times\Omega\\
\rho_t\le1, \, p_t \in H^1_{\rho_t}(\Omega).
\end{cases}
\end{equation}

Thus we found again equation \eqref{main2}, where we imposed also the no-flux boundary condition, which is the one expected when we model a crowd in a closed room $\Om$. Now we are ready to give the formal definition of solution to our problem:

\begin{definition}\label{def:11} Let $ \u \in L^2([0,T] \times \Om)$ and let $\rho_0 \in\cK_1$. Then we define a solution to \eqref{fp1}  to be a couple $(\rho, p ) \in L^{ \infty} ([0,T] \times \Om )\cap AC([0,T];(\sP(\Om),W_2)) \times  L^2 ( [0,T]; H^1( \Om )) $ with $\rho|_{t=0}=\rho_0$ and such that:

\begin{itemize}

\item for all $\phi \in C_c^{\infty} ( \R^d)$ and for all $0\leq r<s \leq T$ we have

$$ \int_r^s \int_{\Om}  (\u_t (x)- \nabla p (x))  \cdot \nabla \phi (x)  \rho_t(x) \, dx \, dt = \int_{\Om} \phi(x)  \rho_s (x) \, dx- \int_{\Om} \phi(x)  \rho_r (x) \, dx; $$
\item we have $0 \leq \rho \leq 1$ for $  \sL^1\mres{[0,T]} \otimes \sL^d\mres{\Om} $-a.e. $(t,x)$ and $p_t \in H^1_{\rho_t} ( \Omega)$ for $ \sL^1\mres{[0,T]}$-a.e. $t$.

\end{itemize}

\end{definition}

\begin{remark}
We remark that the boundary condition implied by our definition is homogeneous Neumann. Note that we assume for the density $\rho$ to be also an absolutely continuous curve in the space of probability measures equipped with the Wasserstein distance\footnote{see the subsection on optimal transport below and \cite{AmbGigSav} for further details.}. Like this it is meaningful to assume $\rho|_{t=0}=\rho_0.$
\end{remark}

In \cite{MauRouSanVen,MauRouSan1} and \cite{aude:phd} the following regularity hypotheses have been assumed to show the existence result: $\u\in C^1$ or $\u=-\nabla D$ for a $\l-$convex potential $D$ and in both cases no dependence on time.

\subsection{The diffusive case}\label{diffusive}

Recently (see \cite{MesSan}) a second order model for crowd motion has been proposed. In a nutshell, it consists of adding a non-degenerate diffusion to the movement and imposing the density constraint. This leads to a modified Fokker-Planck equation and, with the notations previously introduced, it reads as 

\begin{equation}\label{fokker1}
\begin{cases}
\partial_t\rho_t-\Delta\rho_t+\nabla\cdot(\rho_t\v_t)=0\\
\rho_t\le1,\ \rho|_{t=0}=\rho_0\\
\v_t=P_{\adm(\rho_t)}[\u_t],
\end{cases}
\end{equation}
where $\u_t$ is -- as before -- the desired given velocity field of the crowd. Introducing the pressure gradient in the characterization of the projection, one can write system \eqref{fokker1} as
\begin{equation}\label{fokker2}
\begin{cases}
\partial_t\rho_t+\nabla\cdot(\rho_t \u_t ) =\Delta (p_t +\rho_t)\\
\rho_t\le1,\ \rho|_{t=0}=\rho_0, \, p_t \in H^1_{\rho_t}(\Omega)
\end{cases}
\end{equation}
where, as before we equip the equation with the natural no-flux boundary condition on $\partial\Om$; the rigorous definition of solution is similar to that of Definition \ref{def:11}.  We notice also the fact that we have in particular $p_t+ \rho_t \in \beta_2(\rho_t)$ and so we are in fact solving equation \eqref{eqn:main} with $\beta=\beta_2$.

Under the assumption that $\Om$ is convex and $\u\in L^\infty([0,T]\times\Om)$ it has been shown (see \cite[Theorem 3.1]{MesSan}) that the system \eqref{fokker2} admits a solution $(\rho,p)\in L^\infty([0,T]\times\Om)\times L^2([0,T];H^1(\Om)).$ In addition $[0,T]\ni t\mapsto\rho_t$ is an absolutely continuous curve in the 2-Wasserstein space (see Subsection \ref{OptT}). We direct the reader to \cite{MesSan} for further details on this model.

\subsection{Optimal trasport}\label{OptT}
Here we collect some facts about optimal transport that will be needed in the sequel. Given $X_1$ and $X_2$ two measurable spaces and $T$ a measurable map between them, we say that the Borel measure $\mu_2$ is the {\it push forward} of the Borel measure $\mu_1$ through $T$ and we write $\mu_2 = T_{\sharp} \mu_1$, if $\mu_2(A) = \mu_1(T^{-1}(A))$ for every measurable set $A \subseteq X_2$.

Given two measures $\mu \in \sP(X_1)$ and $\nu \in\sP(X_2)$, we define $\Pi(\mu,\nu)$ as the set of  $\gamma \in \sP(X_1 \times X_2)$ such that $(\pi_1)_{\sharp} \gamma  = \mu$ and $(\pi_2)_{\sharp} \gamma = \nu$, where $\pi_i$ is the projection to the $i$-th coordinate: these measures are called {\it transport plans} between $\mu$ and $\nu$.

A particular example of transport plans is given by the transport maps: whenever we have a measurable $T: X_1 \to X_2$ such that $T_{\sharp} \mu = \nu$ we have that the induced plan $\gamma_T=(\id,T)_{\sharp} \mu$ belongs to $\Pi(\mu, \nu)$.
Let us summarize some well-known facts about optimal transport in the following theorem (see for instance \cite{villani} or \cite{OTAM}). 

\begin{theorem}\label{OT}
Let $\Omega \subset \R^d$ be an open bounded set and let $\mu , \nu \in \sP(\Omega)$. Let us consider the following quantities
\begin{equation}
A(\mu, \nu) = \inf \left\{ \int_{\Omega \times \Omega}  |x - y|^2 \dd \gamma \; : \; \gamma \in \Pi ( \mu, \nu ) \right \}
\tag{P}\label{P}
\end{equation}

\begin{equation}
B(\mu,\nu)=\sup \left\{ \int_{\Omega}  \vphi (x) \dd \mu + \int_{\Omega} \psi (x) \dd \nu \; : \;\vphi,\psi\in C_b(\Om), \vphi (x) + \psi (y) \leq \frac 12 | x-y|^2\; \; \forall \,x,y \in \Omega \right\},\tag{D}\label{D}
\end{equation}
where $C_b(\Om)$ denotes the space of bounded continuous functions on $\Om$.
We will call \eqref{P} the primal and \eqref{D} the dual problem.

\begin{itemize}

\item[(i)] There exists at least a minimizer for the primal problem (the set of minimizers is denoted by $\Pi_o(\mu,\nu)$) and there exists also a maximizer $(\vphi,\psi)$ in the dual problem.
\item[(ii)] $A(\mu,\nu) = 2B(\mu, \nu)$ and we will call $W^2_2(\mu,\nu)$ the common value.
\item[(iii)] We can choose a maximizer $(\vphi,\psi)$ of \eqref{D} such that $\vphi$ and $\psi$ are Lipschitz in $\Omega$ and also such that $\frac 12 |x|^2 -\vphi(x)$ is a convex function on the convex hull of $\Omega$. If $\mu\ll \sL^d$ then its gradient is a map $T(x)=x- \nabla \vphi(x)$ such that $T_{\sharp} \mu = \nu$ and whose associated plan is the unique optimal plan, that is $\Pi_o(\mu,\nu)=\{ \gamma_T \}$.
\item[(iv)] If $\Omega$ is convex, $\sP(\Omega)$ endowed with the Wasserstein distance $W_2$ is a geodesic space and if $\mu\ll \sL^d$ the geodesic between $\mu$ and $\nu$ is unique and it is described by $$[0,1]\ni t\mapsto\mu_t:= (\id + t (T- \id) ) _{\sharp} \mu.$$
\end{itemize}
\end{theorem}
The space of probability measures equipped with the {\it Wasserstein distance} $W_2$ is called the {\it Wasserstein space}. We denote it by $\W_2:=(\sP(\Om),W_2)$.

\vspace{1cm}
Structurally, the following two sections contain our main results. Section \ref{two} is devoted to the uniqueness question for first order models where the main tool is the theory of optimal transport. In Section \ref{three} we investigate the uniqueness issue for second order models, using PDE techniques. Finally, in Section \ref{ss:l1} we discuss an approach that could lead to an $L^1-$contraction result in the first order case as well.

\section{Monotone vector fields in the first order case}\label{two}
Let $\Om\subset\R^d$ be a bounded convex domain. In this section we suppose that the desired velocity  field  $\u:[0,T]\times\Om\to\mathbb{R}^d$ of the crowd is a monotone vector field in $L^2([0,T]\times\Om;\R^d)$, i.e. the following assumption is fulfilled: there exists $\lambda\in\R$ such that for all $t\in[0,T]$ there exists a Borel measurable set $A_t\subseteq\Om$ (possibly depending on $t$) with full measure, i.e. $\sL^d(\Om\setminus A_t)=0$ and
\begin{equation}
\left(\u_t(x)-\u_t(y)\right)\cdot(x-y)\le\l |x-y|^2,\;\;\forall\; x,y\in A_t.
\tag{H1}\label{H1}
\end{equation}

The following contractivity results are not very surprising in the Wasserstein context. In practice we show that (Lemma \ref{positivity})
$$ \{ - \nabla p \; : \;  p \in H^1_{\rho}(\Omega) \} \subseteq \partial_{\mathbb{W}_2} \mathscr{I}_1 ( \rho ), $$
 where $\mathscr{I}_1$ the indicatrix function of $\cK_1$ and $\partial_{\mathbb{W}_2}$ denotes the Wasserstein subdifferential (see \cite{AmbGigSav}): then we exploit the geodesic convexity of $\mathscr{I}_1$ in order to get the contraction properties. However, in order to let the reader understand clearly the proofs we will omit the $\mathbb{W}_2$ technical language. It would be interesting to adapt these tools to more general Hele-Shaw problems. 
 
 Although a first written version of these results is essentially contained in \cite{Mes} (Section 4.3.1), here we simplified and clarified some of the proofs. A key observation is the following lemma (see also Lemma 4.3.13 in \cite{Mes}):

\begin{lemma}\label{positivity}
Let $\Omega$ be a convex bounded domain of $\mathbb{R}^d$ and let
$\rho_0,\rho_1\in\sP(\Omega)$ two absolutely
continuous measures such that $\rho_0\le1$ and $\rho_1\le1$ a.e.
Take a Kantorovich potential $\vphi$ from $\rho_0$ to $\rho_1$
and $p\in H^1_{\rho_0}(\Omega).$  Then
$$\int_{\Omega}\nabla\vphi\cdot\nabla p\dd x=\int_{\Omega}\nabla\vphi\cdot\nabla p\dd\rho_0\ge0.$$
\end{lemma}
To prove this result we consider the following extra lemma:

\begin{lemma}\label{derivative}
Let $\Omega$ be a convex bounded domain of $\mathbb{R}^d$ and let
$\rho_0,\rho_1\in\sP(\Omega)$ two absolutely
continuous measures such that $\rho_0\le 1$ and $\rho_1\le 1$ a.e.
Take a Kantorovich potential $\vphi$ from $\rho_0$ to $\rho_1$
and $p\in H^1(\Omega)$. Let $[0,1]\ni t\mapsto\rho_t$ be the geodesic connecting $\rho_0$ to $\rho_1$, with respect to the $2$-Wasserstein distance $W_2$. Then we have that
$$ \frac { \dd}{\dd t}\Bigg|_{\{t=0\}} \int_{\Omega} p \dd \rho_t  = - \int_{\Omega}\nabla\vphi\cdot\nabla p \dd \rho_0. $$
\end{lemma}

\begin{proof} We know (using the interpolation introduced by R. McCann, see \cite{mccann} or Theorem \ref{OT} (iv)) that $\rho_t = (x-t \nabla \vphi (x) ) _{\sharp} \rho_0$ for all $t\in[0,1]$ and so we have
\begin{equation*}
\begin{split}
\frac { \dd} { \dd t}\Bigg|_{\{t=0\}} \int_{\Omega} p \dd \rho_t & = \lim_{ t \to 0} \int_{\Omega} \frac{ p(x-t \nabla \vphi (x) ) - p(x) }{t} \dd \rho_0(x) \\
& = - \lim_{t \to 0 } \int_{\Omega} \frac 1t \int_0^t  \nabla p ( x- s \nabla \vphi (x) )  \cdot \nabla \vphi (x) \dd s \dd \rho_0(x) \\
& = - \lim_{t \to 0 } \int_{\Omega}  A_t( \nabla p )   \cdot \nabla \vphi \dd \rho_0(x),
\end{split}
\end{equation*}
where the second equality is easy to prove, for fixed $t$, by approximation via smooth functions and, for $t\in[0,1]$, we denoted  by $A_t : L^2(\Omega;\R^d) \to L^2_{\rho_0}(\Omega;\R^d)$ the linear operator
$$ A_t (h)(x) = \frac 1t \int_0^t h(x- s \nabla \vphi (x) ) \dd s. $$
Now as a general fact we will prove that $A_t(h) \to h $ strongly in $L^2_{\rho_0}(\Omega;\R^d)$ as $t\to 0,$ for every $h \in L^2(\Omega;\R^d)$. First of all it is easy to see that $ \| A_t \| \leq 1$: indeed
 
\begin{equation*}
\begin{split}
\int_{\Omega} |A_t (h)|^2 \dd \rho_0 & \leq \frac 1t \int_{\Omega} \int_0^t  |h(x- s \nabla \vphi (x) )|^2 \dd s \dd \rho_0(x) \\
& = \frac 1t  \int_0^t \int_{\Omega} |h|^2 \dd \rho_s(x)\dd s \leq \int_{\Omega} |h|^2 \dd x.
\end{split}
\end{equation*}
Here we used the fact that since $\rho_0,\rho_1 \leq 1$ a.e we have also $\rho_t \leq 1$ a.e. for all $t\in[0,1]$. Now it is sufficient to note that for every $\e >0 $ there exists a Lipschitz function $h_{\e}$ such that $\|h_{\e} - h \|_{L^2} \leq \e$, and so we have
$$ \| A_t (h) - h \|_{L^2_{\rho_0}} \leq \| A_t ( h- h_{\e}) \|_{L^2_{\rho_0}} + \| h-h_{\e}\|_{L^2_{\rho_0}} + \|A_t (h_{\e}) - h_{\e} \|_{L^2_{\rho_0}} \leq 2 \e + t L \| \nabla \vphi \|_{L^2_{\rho_0}}, $$
where $L$ is the Lipschitz constant of $h_{\e}$. Taking now the limit as $t$ goes to $0$ we obtain $\limsup_{t \to 0} \| A_t( h) - h \|_{L^2_{\rho_0}} \leq 2\e$; by the arbitrariness of $\e>0$ we conclude.

Now it is easy to finish the proof, since $\nabla p  \in L^2(\Omega)$ and so
$$\frac { \dd} { \dd t}\Bigg|_{\{t=0\}} \int_{\Omega} p \dd \rho_t  =-\lim_{t \to 0}  \int_{\Omega} A_t(\nabla p) \cdot \nabla \vphi \dd \rho_0=-\int_{\Omega} \nabla p \cdot \nabla \vphi \dd \rho_0.$$
\end{proof}

\begin{proof}[Proof of Lemma \ref{positivity}] Let $[0,1]\ni t\mapsto\rho_t$ be the Wasserstein geodesic between $\rho_0$ and $\rho_1$. We know that $\rho_t \leq 1$ a.e. for all $t\in[0,1]$ and in particular it is true that 
$$\int_{\Omega} p \dd \rho_t \leq \int_{\Omega} p\dd x  = \int_{\Omega} p \dd \rho_0,$$
which means that the function $\ds [0,1]\ni t\mapsto\int_\Om p \dd \rho_t$ has a local maximum in $t=0,$ hence its derivative at $0$ is non-positive.

Given this, the claim is a consequence of Lemma \ref{derivative}.
\end{proof}

We will also need a regularity lemma on the continuity equation: by definition a curve $\rho_t$ satisfies a continuity equation with velocity $\v_t$ if for every $\phi \in C^1_0(\R^d)$\footnote{we recall that for us $C_0^1(\R^d)$ is the closure of $C_c^{\infty}(\R^d)$ with respect to the norm $\| \phi \|_1 = \| \phi \|_{\infty} + \| \nabla \phi \|_{\infty}$ or, equivalently, the set of $\phi \in C^1(\R^d)$  such that $\lim_{ |x| \to \infty} | \phi (x) | + | \nabla \phi (x)| \to 0$} the application $\ds [0,T]\ni t \mapsto \int_{\R^d} \phi \dd  \rho_t$ is absolutely continuous and its derivative is $\ds\int_{\R^d} \v_t \cdot \nabla \phi \dd  \rho_t$ for almost every time $t$. We will prove that under some integrability assumption there exists a universal full-measure set of differentiability, even when $\phi \in H^1(\R^d)$.

\begin{lemma}\label{lem:diffB} Let $\rho_t$ be a weakly continuous curve of probability measures on $\R^d$; let us suppose that the continuity equation $\partial_t \rho_t + \nabla \cdot ( \rho_t \v_t)=0$ holds with a velocity field $\v_t$ such that $\ds\int_0^T\int_{\R^d} | \v_t| \dd  \rho_t \dd t < +\infty$. Then there exists a set $\Tau \subset [0,T]$ such that $ \mathscr{L}^1 ( [0,T] \setminus \Tau ) =0$ and
\begin{equation}\label{eq:deriv} \lim_{ h \to 0} \frac 1h \left( \int_{\R^d} \phi \dd  \rho_{t+h} - \int_{\R^d} \phi \dd  \rho_t \right) = \int_{\R^d} \nabla \phi \cdot \v_t \dd  \rho_t \qquad \forall t \in \Tau
\end{equation}
for every $\phi \in C_0^1(\R^d)$. Moreover if we further have that $\rho_t \leq 1$ a.e. and $\ds\int_0^T\int_{\R^d} | \v_t|^2 \dd  \rho_t \dd t < +\infty$ then we can require that \eqref{eq:deriv} also holds for every $\phi \in H^1(\R^d)$.
\end{lemma} 

\begin{proof} Let us prove the following general statement: for a given separable Banach space $B$ and a curve $x^*  \in L^1([0,T];B^*)$, there exists  $\Tau \subset [0,T]$ such that $\Tau$ is a set of Lebesgue points for the map $t \mapsto x^*_t(b)$ for every $b \in B$, and moreover $\mathscr{L}^1 ( [0,T] \setminus \Tau )=0$. This can be proven easily by choosing a dense subset $(b_n) \subset B$ and then taking $\Tau_n$ as the set of Lebesgue points of $t \mapsto x^*_t(b_n)$ and $\Tau_0$ as the Lebesgue points of $t \mapsto \| x^*_t \|$, and then we take $ \Tau = \bigcap_{ n \geq 0} \Tau_n$. For every $b \in B$ and $\e>0$ let us consider $i\in\mathbb{N}$ such that $\|b_i - b\| \leq \e$ and then taking $t \in \Tau$ we have
$$  \frac 1{2\delta} \int_{t-\delta}^{t+\delta}| x^*_s(b) - x^*_t(b)| \dd s  \leq \e \Bigl( \|x^*_t\|  + \frac 1{2\delta}\int_{t-\delta}^{t+\delta} \| x^*_s\| \dd s \Bigr)  +   \frac 1{2\delta} \int_{t-\delta}^{t+\delta} |x^*_s(b_i)  - x^*_t(b_i) | \dd  s. $$
Now, taking the limit as $\delta \to 0$ and using the properties of $\Tau$ we get
$$ \limsup_{\delta \to 0}   \frac 1{2\delta} \int_{t-\delta}^{t+\delta} |x^*_s(b)  - x^*_t(b)| \dd  s \leq 2 \e \| x^*_t\|$$
and by the arbitrariness of $\e$ we conclude that $t$ is a Lebesgue point for $x^*_s(b)$. Now it is easy to conclude thanks to the fact that we know that
$$ \int_{\R^d} \phi \dd  \rho_{t+h} - \int_{\R^d} \phi \dd  \rho_t = \int_t^{t+h} \int_{\R^d} \nabla \phi \cdot \v_s \dd  \rho_s, $$
and noticing that $\ds x^*_s : \phi \mapsto \int_{\R^d} \nabla \phi \cdot \v_s \dd  \rho_s$ satisfies the assumption $x^* \in L^1([0,T];B^*)$, when $B$ is the Banach space $C^1_0(\R^d)$ in the first case and  $H^1(\R^d)$ in the second case. Notice that if we follow the construction of $\Tau$ for the case $B=C^1_0(\R^d)$, this set of times will work also for $H^1(\R^d)$ since $C^1_0$ is dense in $H^1$.
\end{proof}

Now we are in position to prove the main theorem of this section, namely:
\begin{theorem}\label{thm:unique_first}
Suppose $\Omega\subset\mathbb{R}^d$ is a bounded convex domain,
$\u$ is a vector field satisfying Assumption \eqref{H1} and let $\rho_0\in\cK_1$ be an admissible
initial density.  Let us suppose that there exist $(\rho^1,p^1),(\rho^2,p^2)$ two solutions to the system
\begin{equation}\label{fp2}
\left\{\begin{array}{ll}
\partial_t\rho_t+\nabla\cdot(\rho_t(\u_t-\nabla p_t))=0 & {\rm{in\ }}(0,T)\times\Om\\[7pt]
\rho_t \le 1,\; p_t\ge0,\; (1-\rho_t)p_t=0 &{\rm{a.e.\ in\ }}[0,T]\times\Om, \\[7pt]
\rho|_{t=0}=\rho_0,
\end{array} \right.
\end{equation}
$p^i\in L^2( [0,T]; H^1(\Omega))$ and $\rho_i \in \sP (\Om)$ for $i\in\{1,2\}$, where the first equation is supposed to be satisfied in duality with $C^{\infty}_c(\R^d)$ (see Definition \ref{def:11}) in order to take into account the boundary conditions. Then, $\rho^1=\rho^2$ and $p^1=p^2$ a.e. In particular, under the same assumptions, we can say that there exists a unique pair $(\rho,\v)$ that solves \eqref{fp0}.
\end{theorem}

\begin{proof}

We associate to the two curves $\rho^1_t$ and $\rho^2_t$ a continuity equation \eqref{fp1} with the corresponding vector fields $\v^1_t=\u^1_t- \nabla p^1_t$ and $\v^2_t=\u^2_t- \nabla p^2_t $. Let us compute and estimate
$\ds\frac{\dd}{\dd t}\frac{1}{2}W_2^2(\rho^1_t,\rho^2_t)$; we refer also to \cite[Theorem 8.4.7]{AmbGigSav} for a more general statement, but we prefer to include a simpler proof in this case. We know that $t \mapsto W_2^2(\rho^1_t, \rho^2_t)$ is absolutely continuous: let us consider a time $t$ for which its derivative exists and also such that $t \in \Tau^1 \cap \Tau^2$, where $\Tau^i$ is the set for which \eqref{eq:deriv} is satisfied for the continuity equation for $\rho^i$. Then we know that, for all $s \in [0,T]$ we have 
\begin{equation}\label{eq:Kdua} \frac12 W_2^2 (\rho^1_s,\rho^2_s) \geq \int_{\Om} \vphi_t \dd  \rho^1_s + \int_{\Om} \psi_t \dd  \rho^2_s,
\end{equation}
where $(\vphi_t,\psi_t)$ is a pair of Kantorovich potentials for $\rho^1_t$ and $\rho^2_t$. In particular we have equality in \eqref{eq:Kdua} for $t=s$ and so, since both sides are differentiable by hypothesis, their derivatives are equal. Hence we get
$$\frac{\dd}{\dd t}\frac{1}{2}W_2^2(\rho^1_t,\rho^2_t) = \int_{\Omega}\v^1_t \cdot \nabla \vphi_t \dd\rho^1_t +  \int_{\Omega}\v^2_t \cdot \nabla \psi_t \dd\rho^2_t,$$
 where we used the fact that, since $\Omega$ is bounded,  we can assume $\vphi_t$ and $\psi_t$ to be Lipschitz and in particular they belong to $H^1$, which allows to use Lemma \ref{lem:diffB}.

We also know that thanks to Theorem \ref{OT} there is a pair of optimal
transport maps $T_t(x)=x-\nabla \vphi_t(x)$ and $S_t(y)=y - \nabla \psi_t(y)$ such that $(T_t)_\sharp\rho^1_t=\rho^2_t$ and $(S_t)_\sharp \rho^2_t = \rho^1_t$ for all $t\in[0,T]$; moreover $T_t$ is the inverse of $S_t$ (in the appropriate almost everywhere sense). Using this, we have $\nabla \psi_t(T_t(x)) = -\nabla \vphi_t(x)$ and in particular using the change of variable formula $y=T_t(x)$ we get
$$ \int_{\Omega} \nabla \psi_t (y) \cdot \u_t (y) \dd \rho_t^2(y) = - \int_{\Omega} \nabla \vphi_t(x) \cdot \u_t ( T_t(x)) \dd \rho_t^1(x).$$

We can use this to split the  formula for the derivative of $W^2_2(\rho^1_t,\rho^2_t)/2$. We use the result of Lemma \ref{positivity} and then rewrite the term regarding $\u_t$ in terms of transport maps and see:
\begin{align*}
\frac{\dd}{\dd t}\frac{1}{2}W_2^2(\rho^1_t,\rho^2_t)&=\int_{\Omega}\v^1_t \cdot \nabla \vphi_t \dd\rho^1_t +  \int_{\Omega}\v^2_t \cdot \nabla \psi_t \dd\rho^2_t \\
&= \int_{\Omega}\u_t \cdot \nabla \vphi_t \dd\rho^1_t +  \int_{\Omega}\u_t \cdot \nabla \psi_t \dd\rho^2_t  - \int_{\Omega}\nabla p^1_t \cdot \nabla \vphi_t \dd\rho^1_t -  \int_{\Omega}\nabla p^2_t \cdot \nabla \psi_t \dd\rho^2_t \\
& \leq \int_{\Omega} \nabla \vphi_t(x) \cdot \bigl[\u_t (x) - \u_t(T_t(x))\bigr] \dd \rho^1_t  \\ 
& \leq \int_{\Omega} (x-T_t(x)) \cdot \bigl[\u_t (x) - \u_t(T_t(x))\bigr] \dd \rho^1_t \\
& \leq \lambda \int_{\Omega} | x- T_t(x)|^2 \dd \rho^1_t \leq \lambda W_2^2(\rho^1_t, \rho^2_t).
\end{align*}

%
Gr\"onwall's lemma implies that
$$W_2^2(\rho^1_t,\rho^2_t)\le e^{2\lambda t}W_2^2(\rho^1_0,\rho^2_0).$$
Since $\rho_0^1=\rho_0^2=\rho_0$ a.e., the above property implies that $\rho^1=\rho^2$ a.e. in $[0,T]\times\Om.$ From this fact we can easily deduce that $\Delta(p_t^1-p_t^2)=0,$ for a.e. $t\in[0,T]$ in the sense of distributions. In particular $p_t^1-p_t^2$ is analytic in the interior of $\Om.$ Moreover, both $p^1_t$ and $p^2_t$ vanish a.e. in the set $\{\rho_t^1<1\}$ which has a positive Lebesgue measure greater than $|\Om|-1>0.$ Thus, $p^1=p^2$ a.e. in $[0,T]\times\Om.$ The claim follows.
\end{proof}


\begin{remark}
The existence result for system \eqref{fp2} was obtained in different settings in the literature. On the one hand, if $\u=-\nabla D$ (for a reasonably regular potential $D$), the existence and uniqueness of a pair $(\rho,p)$ can be obtained by gradient flow techniques in $\W_2(\Om)$ (see \cite{AmbGigSav, MauRouSan1, aude:phd}), under the assumption that $\u$ is monotone that translates into a $\lambda$-convexity for $D$. On the other hand, if $\u$ is a general field with $C^1$ regularity, the existence result is proven with the help of a well-chosen splitting algorithm (see \cite{MauRouSanVen, aude:phd}). 

Nevertheless, combining the techniques developed in \cite{MesSan} on the one hand, and the well-known DiPerna-Lions-Ambrosio theory on the other hand, we expect to obtain existence result for \eqref{fp2} for more general vector fields with merely Sobolev regularity and suitable divergence bounds.
\end{remark}

\begin{remark}
The monotonicity assumption \eqref{H1} is not surprising in this setting. We remark that the same assumption was required in \cite{NatPelSav} to prove  contraction properties for a general class of transport costs along the solution of the Fokker-Planck equation in $\R^d$
$$\partial_t\rho-\Delta\rho+\diver(B\rho)=0,\;\;\;\rho|_{t=0}=\rho_0,$$
where the velocity field $B:\R^d\to\R^d$ was supposed to satisfy the monotonicity property \eqref{H1}.
\end{remark}

We note also the fact that we can allow for moving domains $\Omega_t$ (considering always a no-flux boundary condition): in fact our proof never uses that the domains are fixed but uses only convexity at any fixed time. This generalization has been used in \cite{DiMoSa} for proving uniqueness for an evolution equation with density constraint driven only by the boundary of the moving sets.

\section{Bounded vector field in the diffusive case}\label{three}
We use Hilbert space techniques (similarly to the one developed in \cite{Cro, Bou, PetQuiVaz}; see also Section 3.1. from \cite{Por}) to study the uniqueness of a solution of the diffusive crowd motion model with density constraints described in Subsection \ref{diffusive} (see also \cite{MesSan}). Moreover we can expect that this holds under more general assumptions in the presence of a non-degenerate diffusion in the model. 

Let  $\u:[0,T]\times\Om\to\R^d$ be a given vector field, which represents again the desired velocity field of the crowd, $\Om\subset\R^d$ a  bounded open set with $C^1$ boundary, $\rho_0\in\sP(\Om)$ the initial density of the population such that $0\le\rho_0\le 1$ a.e. in $\Om$ and let us consider the following problem
\begin{equation}\label{2nd_1}
\left\{
\begin{array}{rl}
\partial_t\rho_t-\Delta\rho_t+\diver(P_{\adm(\rho_t)}[\u_t]\rho_t)=0, & \iin\ (0,T)\times\Om;\\[10pt]
\rho|_{t=0}=\rho_0,\;\;\; 0\le \rho_t\le 1,  & \ae\ \iin\ \Om,
\end{array}
\right.
\end{equation} 
equipped with the natural no flux boundary condition.
Introducing the pressure variable, equivalently the above system can be written as
\begin{equation}\label{2nd_2}
\left\{
\begin{array}{rl}
\partial_t\rho_t-\Delta\rho_t-\Delta p_t+\diver(\u_t\rho_t)=0, & \iin\ (0,T)\times\Om,\\[10pt]
\rho|_{t=0}=\rho_0, & \iin\ \Om\\[10pt]
(\nabla\rho_t+\nabla p_t -\u_t\rho_t)\cdot\n = 0, & \oon\ \partial\Om,\ {\rm{for\ a.e.}}\ t\in[0,T],
\end{array}
\right.
\end{equation} 
for a pressure field $p_t\in H^1_{\rho_t}(\Om)$.
It has been shown in \cite{MesSan} that under the assumption that 
\begin{equation}\label{eq:H2}
\u\in L^{\infty}([0,T]\times\Om;\R^d)\tag{H2}
\end{equation}
the systems \eqref{2nd_1} and \eqref{2nd_2} have a solution. More precisely there exist an absolutely continuous curve $[0,T]\ni t\mapsto\rho_t\in \W_2$ and  $p_t\in H^1_{\rho_t}(\Om)$ for a.e. $t\in[0,T]$ (in particular $\rho\in L^\infty([0,T]\times\Om)$ and $p\in L^2([0,T];H^1(\Om))$ such that $(p,\rho)$ solves \eqref{2nd_2} in weak sense (see \eqref{2nd_dual}). 

Our aim in this section is to show the following theorem: 
\begin{theorem}
Let $\u$ satisfy \eqref{eq:H2}. Then there exists a unique pair $(\rho,p)\in L^\infty([0,T]\times\Om)\times L^2([0,T]; H^1(\Om))$ that solves \eqref{2nd_2} in the weak sense \eqref{2nd_dual}. Moreover for every solution $(\rho^1,p^1)$ and $(\rho^2,p^2)$ to  \eqref{2nd_2} we have the $L^1-$contraction property
$$ \int_{\Om} | \rho^1(T,x) - \rho^2(T,x)| \dd x \leq \int_{\Om} |\rho^1(0,x) - \rho^2(0,x)| \dd x.$$
\end{theorem}

\begin{proof}
The existence of a solution $(\rho,p)$ was obtained in \cite{MesSan}. Now let us show the uniqueness of the solution via an $L^1-$contraction property; notice that this contraction is valid for a general open set $\Om$ with $C^1$ boundary.

Let us write the weak formulation of \eqref{2nd_2}: for every smooth test function  $\phi:[0,T]\times\Om\to\R$ with $\nabla\phi\cdot\n=0$ on $[0,T]\times\partial\Om$ we have

\begin{equation}\label{2nd_dual}
\int_0^T\int_\Om\left[\rho\partial_t\phi+(\rho+p)\Delta\phi+\rho\u\cdot\nabla\phi\right]\dd x\dd t+\int_\Om\rho_0(x)\phi(0,x)\dd x=\int_\Om\rho(T,x)\phi(T,x)\dd x.
\end{equation}
By density arguments the above formulation holds for $\phi\in W^{1,1}([0,T];L^1(\Om))\cap L^2([0,T];H^2(\Om)).$

Now, let us consider two solutions $(\rho^1,p^1)$ and $(\rho^2,p^2)$ of Problem \ref{2nd_2}, with initial conditions respectively $\rho^1_0$ and $\rho^2_0$. 
 Writing the weak formulation \eqref{2nd_dual} for both of them and taking the difference we obtain
\begin{equation}\label{difference}
\mathcal{I}(\phi,T) =\mathcal{I}(\phi,0) + \int_0^T\int_\Om\left[(\rho^1-\rho^2)\partial_t\phi+(\rho^1-\rho^2+p^1-p^2)\Delta\phi+(\rho^1-\rho^2)\u\cdot\nabla\phi\right]\dd x\dd t,
\end{equation}
where $\ds\mathcal{I}(\phi,t)= \int_{\Om} \phi(t,x) [ \rho^1(t,x) - \rho^2(t,x)]\dd x$. We introduce the following quantities
$$
A:=\frac{\rho^1-\rho^2}{(\rho^1-\rho^2)+(p^1-p^2)}\;\;\;{\rm{and}}\;\;\; B:=\frac{p^1-p^2}{(\rho^1-\rho^2)+(p^1-p^2)}.
$$
Note that $0\le A\le 1$ and $0\le B\le 1$ a.e. in $[0,T]\times\Om$ and $A+B=1.$ To be consistent with these bounds, we set $A=0$ when $\rho^1=\rho^2$, even if $p^1=p^2$ and $B=0$ when $p^1=p^2$, even if $\rho^1=\rho^2$. With these notations the weak formulation for the difference gives
\begin{equation}\label{uniq_cond}
\mathcal{I}(\phi,T) = \mathcal{I}(\phi,0) + \int_0^T\int_\Om((\rho^1-\rho^2)+(p^1-p^2))\left[A\partial_t\phi+(A+B)\Delta\phi+A\u\cdot\nabla\phi\right]\dd x\dd t.
\end{equation}
For a smooth function $\theta:\Om\to\R$ such that $|\theta| \leq 1$, let us consider the dual problem
\begin{equation}\label{parabolic1} 
\left\{
\begin{array}{rl}
A\partial_t\phi+(A+B)\Delta\phi+A\u\cdot\nabla\phi=0, & \iin\ [0,T)\times\Om,\\[10pt]
\nabla\phi\cdot\n=0\ \oon\ [0,T]\times\partial\Om, &  \phi(T,\cdot)=\theta\ \ae\ \iin\ \Om.
\end{array}
\right.
\end{equation}

Let us remark that if we are able to find a (reasonably regular) solution $\phi$ for this problem for any $\theta$ smooth then, using the maximum principle and then optimizing in $\theta$, we would get an $L^1-$contraction result for $\rho$: in particular we get uniqueness for the initial value problem for $\rho$ and hence also for $p$ (as done in the end of the proof of Theorem \ref{thm:unique_first}). However, since the coefficients in \eqref{parabolic1} are not regular,  we study a regularized problem. For $\e>0$ let us consider $A_\e$, $B_\e$, $\u_{\e}$ to be smooth approximations of $A$, $B$ and $\u$ such that
\begin{equation}\label{eq:approx}
\|A-A_\e\|_{L^r([0,T]\times\Om)} + \|B-B_\e\|_{L^r([0,T]\times\Om)} + \| \u - \u_{\e} \|_{L^r([0,T] \times \Om )}<\e,\;\; ,\;\; \e<A_\e, B_\e\le 1,
\end{equation}
for some $1\le r<+\infty$, the value of which to be chosen later. 
The regularized problem reads as follows
\begin{equation}\label{parabolic2} 
\left\{
\begin{array}{rl}
\partial_t\phi_\e+(1+B_\e/A_\e)\Delta\phi_\e+\u_{\e}\cdot\nabla\phi_\e=0, & \iin\ (0,T)\times\Om,\\[10pt]
\nabla\phi_\e\cdot\n=0\ {\rm{a.e.}}\ \oon\ [0,T]\times\partial\Om, &  \phi_\e(T,\cdot)=\theta\ \ \iin\ \Om.
\end{array}
\right.
\end{equation}
For all $\e>0$ the above problem is uniformly parabolic and $B_\e/A_\e$ is continuous and positive. Moreover  $\theta$ is smooth, thus by classical results (see for instance \cite{lady, krylov}) the problem has a (unique) solution $\phi_\e\in C^1([0,T] \times \overline{ \Om} )$. In particular $\phi_\e$  can be used as test function in \eqref{2nd_dual}. In the followings we shall use some standard uniform estimates (in $\e$)  on $\phi_\e$ given in Lemma \ref{lem:estimates}.

In particular, using $\phi_\e$ as test function in \eqref{uniq_cond} one has
\begin{align*}
\mathcal{I}(\phi_{\e},T) - \mathcal{I}(\phi_{\e},0)   & = \int_0^T\int_\Om (\rho^1-\rho^2+p^1-p^2)\left[A\partial_t\phi_\e+(A+B)\Delta\phi_\e+A\u\cdot\nabla\phi_\e\right]\dd x\dd t\\
&=\int_0^T\int_\Om (\rho^1-\rho^2+p^1-p^2)\left[A\partial_t\phi_\e+(A+B)\Delta\phi_\e+A\u\cdot\nabla\phi_\e\right]\dd x\dd t\\
&-\int_0^T\int_\Om (\rho^1-\rho^2+p^1-p^2)A\left[\partial_t\phi_\e+(1+B_\e/A_\e)\Delta\phi_\e+\u_\e \cdot\nabla\phi_\e\right]\dd x\dd t\\
&=\int_0^T\int_\Om (\rho^1-\rho^2+p^1-p^2)(B_\e/A_\e)(A_\e-A)\Delta\phi_\e\dd x\dd t\\
&+ \int_0^T\int_\Om (\rho^1-\rho^2+p^1-p^2)(B-B_\e)\Delta\phi_\e\dd x\dd t\\
&+ \int_0^T\int_\Om (\rho^1-\rho^2+p^1-p^2)A(\u_\e - \u) \cdot \nabla \phi_\e\dd x\dd t\\
&:=I_\e^1+I_\e^2 + I_\e^3
\end{align*}
Let us show that $|I_\e^1|\to 0$, $|I_\e^2|\to 0$ and $|I_{\e}^3| \to 0$ as $\e\to 0$. More precisely, first let us recall that  $0\le \rho^1,\rho^2\le 1$ a.e. in $[0,T]\times\Om$, hence $\rho^1,\rho^2\in L^\infty([0,T]\times\Om)$. On the other hand $p^1,p^2\in L^2([0,T]; H^1(\Om))$ and by Corollary \ref{cor:proj} we have that 
$$\int_\Om|\nabla p^i_t|^2\dd x\le \int_\Om|\u_t|^2\dd t,$$ 
for almost every $t\in[0,T].$ This implies that (since $\u$ is bounded)
$${\rm{ess-sup}}_{t\in[0,T]}\|\nabla p^i_t\|_{L^2(\Om)}\le C.$$
In addition, $p^i$'s being pressures one has $|\{p_t^i=0\}|\ge|\{\rho_t^i<1\}|\ge |\Om|-1>0$ for a.e. $t\in[0,T],$ and so by a suitable version of Poincar\'e's inequality (since $p_t^i$ vanishes on a set of positive Lebesgue measure) one obtains $p^i\in L^\infty([0,T]; H^1(\Om)).$ In particular, by the Sobolev embedding this implies that $p^i\in L^\infty([0,T];L^q(\Om)),$ $i\in\{1,2\}$ for all $1\le q\le2^*.$ Let us fix $q\in(2,2^*),$  where $2^{*}=2d/(d-2)$ if $d\ge 3$ and $2^*=\infty$ if $d=2$. 

This implies the following estimates
\begin{align*} 
|I_\e^1|&\le \|\rho^1-\rho^2\|_{L^{\infty}([0,T]\times\Om)}\cdot\|(B_\e/A_\e)^{1/2}(A-A_\e)\|_{L^2([0,T]\times\Om)}\cdot \|(B_\e/A_\e)^{1/2}\Delta\phi_\e\|_{L^2([0,T]\times\Om)}\\
&+\int_0^T\|p^1-p^2\|_{L^{q}(\Om)}\cdot\|(B_\e/A_\e)^{1/2}(A-A_\e)\|_{L^r(\Om)}\cdot \|(B_\e/A_\e)^{1/2}\Delta\phi_\e\|_{L^2(\Om)}\dd t\\
&\le C(1/\e)^{1/2}\e\\
&+\|p^1-p^2\|_{L^\infty(L^{q})}\cdot\|(B_\e/A_\e)^{1/2}(A-A_\e)\|_{L^2(L^r)}\cdot \|(B_\e/A_\e)^{1/2}\Delta\phi_\e\|_{L^2(L^2)} \\
&\le C\e^{1/2}\to 0, \;\; {\rm{as}\ }\e\to 0
\end{align*}
and similarly 
\begin{align*} 
|I_\e^2|&\le \|\rho^1-\rho^2\|_{L^{\infty}([0,T]\times\Om)}\cdot\|(A_\e/B_\e)^{1/2}(B-B_\e)\|_{L^2([0,T]\times\Om)}\cdot \|(B_\e/A_\e)^{1/2}\Delta\phi_\e\|_{L^2([0,T]\times\Om)}\\
&+\|p^1-p^2\|_{L^\infty(L^{q})}\cdot\|(A_\e/B_\e)^{1/2}(B-B_\e)\|_{L^2(L^r)}\cdot \|(B_\e/A_\e)^{1/2}\Delta\phi_\e\|_{L^2(L^2)}\\
&\le C(1/\e)^{1/2}\e= C\e^{1/2}\to 0, \;\; {\rm{as}\ }\e\to 0;
\end{align*}
finally we have
$$ |I_\e^3| \leq  \| \rho^1 - \rho^2\|_{L^{\infty}(L^{q})} \cdot \| \u - \u_{\e} \|_{L^{1}(L^r)} \cdot \| \nabla \phi_\e \|_{L^{\infty}(L^2)} \leq C \e \to 0,$$
 where $r>1$ is the exponent such that $\ds\frac12+\frac1r+\frac{1}{q}=1,$ i.e. $\ds r=\frac{2q}{q-2}$ (recall that $q\in(2,2^*)$. We used the approximations \eqref{eq:approx} with this specific $r$. Hence we obtained that
$$\mathcal{I}(\phi_{\e},T ) = \mathcal{I}(\phi_\e , 0) + o(1);$$
now we use the fact that we fixed the final condition $\theta$ for all $\e$, and also that $|\phi_{\e} | \leq 1$, that is implied by the strong maximum principle (see for example Theorem 2.9 in \cite{lieberman}); in particular we can write
$$ \int_{\Om} \theta(x) \bigl( \rho^1 (T,x )  - \rho^2(T,x) \bigr) \dd x = \int_{\Om} \phi_{\e}(0,x) ( \rho_0^1 - \rho_0^2) \dd x + o(1) \leq \int_{\Om} | \rho^1_0 - \rho^2_0| \dd x + o(1).$$
Letting $\e \to 0$ we finally find $\ds\int_\Om \theta ( \rho^1_T - \rho^2_T)\dd x \leq \| \rho^1_0 - \rho^2_0 \|_{L^1}$, and so, optimizing in $|\theta| \leq 1$, we proved the $L^1-$contraction. In particular this implies that $\rho^1=\rho^2$ a.e. in $[0,T]\times\Om$ whenever $\rho^1_0 = \rho^2_0$. 
Then one obtains also that $p^1=p^2$ a.e. in $[0,T]\times\Om$, as in the end of the proof of Theorem \ref{thm:unique_first}. The result follows.

\end{proof}

\begin{lemma}\label{lem:estimates}
Let $\phi_\e$ be a solution of \eqref{parabolic2}. Then there exists a constant 
$$C=C(T,\|\u\|_{L^\infty},\|\nabla \theta\|_{L^2(\Om)})>0$$ 
such that we have the following estimates, uniformly in $\e>0$:
\begin{itemize}
\item[(i)] $\ds\sup_{t\in[0,T]}\|\nabla\phi_\e(t)\|_{L^2(\Om)}\le C;$\\
\item[(ii)] $\|(B_\e/A_\e)^{\frac12}\Delta\phi_\e\|_{L^2([0,T]\times\Om)}\le C;$\\
\item [(iii)] $\|\Delta\phi_\e\|_{L^2([0,T]\times\Om)}\le C.$
\end{itemize}
\end{lemma}

\begin{proof} 
Let us multiply the first equation in \eqref{parabolic2} by $\Delta\phi_\e$ and integrate over $[t,T]\times\Om$ for $0\le t<T$. We obtain
\begin{align}\label{estim1}
\frac12\|\nabla\phi_\e(t)\|^2_{L^2(\Om)}&+\int_t^T\int_\Om(1+B_\e/A_\e)|\Delta\phi_\e|^2\dd x\dd t\nonumber\\
&=\frac 12 \int_\Om| \nabla\theta|^2 \dd x -\int_t^T\int_\Om \u_\e \cdot\nabla\phi_\e\Delta\phi_\e\dd x\dd t
\end{align}
Hence by Young's inequality we have 
\begin{align*}
\frac12\|\nabla\phi_\e(t)\|^2_{L^2(\Om)}&\le \frac{\|\u\|_{L^\infty}}{2\d}\int_t^T\|\nabla \phi_\e(s)\|^2_{L^2(\Om)}\dd s+\frac12\|\nabla \theta\|^2_{L^2(\Om)}\\
&\le C+\frac{C}{2}\int_t^T\|\nabla\phi_\e(s)\|_{L^2(\Om)}^2\dd s
\end{align*}
where the term in $|\Delta \phi_\e|^2$ has been absorbed by the left hand side, and $0<\d\le2/\|\u\|_{L^\infty}$ is a fixed constant and the constant $C>0$ is depending just on $\|\nabla \theta\|_{L^2(\Om)}$ and $\|\u\|_{L^\infty}$. Hence by Gr\"onwall's inequality we obtain 
$$\frac12\|\nabla\phi_\e(t)\|^2_{L^2(\Om)}\le Ce^{C(T-t)},$$
which implies in particular that $\ds\sup_{t\in[0,T]}\|\nabla\phi_\e(t)\|_{L^2(\Om)}\le C.$ Thus ${\rm{(i)}}$ follows. 

On the other hand choosing $\d:=2/\|\u\|_{L^\infty}$ in Young's inequality used in \eqref{estim1} and using (i), we obtain
$$\int_t^T\int_\Om(B_\e/A_\e)|\Delta\phi_\e|^2\dd x\dd t\le C$$
hence $\|(B_\e/A_\e)^{\frac12}\Delta\phi_\e\|_{L^2([0,T]\times\Om)}\le C,$ and thus  ${\rm{(ii)}}$ follows.

By \eqref{estim1}, ${\rm{(i)}}$ and ${\rm{(ii)}}$ easily imply ${\rm{(iii)}}.$
\end{proof}

\section{About the $L^1-$contraction in the first order case}\label{ss:l1}
%
%

In the previous section we proved uniqueness in the second order case using an $L^1-$contraction result. We expect this result to be true also in the first order case; however we expect the treating of this $L^1-$contraction problem in the most general framework (let us say when $\u \in W^{1,1}$) to be difficult, since this approach should ``include'' (in the proof) also the well-posedness in the Di Perna-Lions theory as a special case. In fact, when one has $\nabla \cdot \u=0$ and if one starts with $\rho_0 \leq 1$ then this condition is preserved without adding the pressure term.\\

However in the case treated in Section \ref{two} (i.e. with a monotone velocity field) we can try to sketch a proof of the $L^1$ contraction result using the uniqueness already proved: let us approximate the solution  discretizing in time using the splitting discrete scheme ``continuity equation + Wasserstein projection onto the set $\{ \rho \leq 1\}$'' (similarly to the scheme introduced in \cite{MesSan}). Since the vector field is monotone there exists a unique flow, that implies also that there exists a unique solution to the initial value problem $\partial_t \rho_t + \nabla \cdot ( \u_t \rho_t)=0$, $\rho_{t_0}=\rho$; in particular one can define a function $\Psi_{t_0}^{t_1} (\rho) := \rho_{t_1}$, that will satisfy also the semi-group rule $\Psi_t^{s'} \circ \Psi_s^t = \Psi_s^{s'}$. Let $\tau$ be a time step and define recursively $\rho_0^{\tau}= \rho_0$ and then

$$\rho^{\tau}_{n+1} =\begin{cases} \Psi_{n\tau}^{(n+1)\tau}( \rho^{\tau}_n) \qquad &\text{ if } n \text{ is even } \\ \mathcal{P} ( \rho_{n} )& \text{ if }n\text{ is odd.}\end{cases}$$
Then the $L^1$ distance is preserved through the continuity equation step while it decreases after the projection thanks to Lemma \ref{lem:contraction}. So in particular the $L^1-$contraction is true in the discrete scheme; once we have that the scheme converges as $\tau \to 0$ to our equation, the uniqueness result gives that this property is preserved in the limit. In order to guarantee convergence (see \cite{DiMoSa} for general conditions) the crucial quantity to estimate is $W_2(\Psi_t^{t+\tau} (\rho), \rho)$; using the Benamou-Brenier formula, two conditions that guarantee a good estimate are:

\begin{itemize}
\item $\u\in L^2([0,T];L^{\infty}(\R^d))$: in this case we would have $ W_2^2(\Psi_t^{t+\tau} (\rho), \rho) \leq  \t\int_t^{t+\tau}  \| \u_s \|^2_{L^\infty}\dd s$;

\item $(\nabla \cdot \u)_{-} \in L^1([0,T] ; L^{\infty}(\R^d))$: in this case we would have $\| \Psi_t^s (\rho) \|_{L^\infty} \leq C \| \rho \|_{L^\infty}$ for some universal $C$ and in particular we have $W_2^2(\Psi_t^{t+\tau} (\rho), \rho) \leq C  \| \rho \|_{L^\infty}  \t\int_t^{t+\tau} \| \u_s \|^2_{L^2}\dd s$.
\end{itemize}

We believe that this general scheme (uniqueness in the Wasserstein framework and approximation with $L^1-$contractive time discrete approximations) could be adapted also with some other convection terms $\Phi$.

\begin{lemma}\label{lem:contraction} Let us consider the projection operator $\mathcal{P}: \sP(\R^d ) \to \sP(\Omega)$
\begin{equation}\label{def:proj} \mathcal{P}(\rho) = {\rm argmin} \{ W_2^2(\rho, \eta) \; : \; \eta \in \sP(\Omega) , \; \eta \leq 1 \}.
\end{equation}
Then, when $\rho_1, \rho_2$ are probability densities we have $\| \mathcal{P}(\rho_1) - \mathcal{P} ( \rho_2) \|_{L^1} \leq \|\rho_1 - \rho_2\|_{L^1}$.
\end{lemma}
\begin{proof}
First we see that formula \eqref{def:proj} let us extend $\mathcal{P}$ to measures in $\sM_+(\R^d)$ with mass less than $|\Omega|$. In this context a monotonicity property is true (see for example Theorem 5.1 in \cite{AleKimYao}), that is we have $\mathcal{P}(\rho) \leq \mathcal{P}(\eta)$ almost everywhere if $\rho \leq \eta$. Now we can derive the $L^1$ contraction: let us denote $\rho= \min \{ \rho_1 , \rho_2\}$. First of all we have that 
$$|\rho_1 - \rho_2| = \max\{\rho_1, \rho_2\} - \min \{ \rho_1 , \rho_2\} = ( \rho_1 + \rho_2 - \min\{\rho_1, \rho_2\} ) - \min\{\rho_1,\rho_2\};$$
in particular $\| \rho_1 - \rho_2\|_{L^1} = 2- 2 \int \rho\dd x$. By the monotonicity we have $\mathcal{P}(\rho_i) \geq \mathcal{P} (\rho)$ and in particular we have $\min\{ \mathcal{P}(\rho_1) , \mathcal{P}(\rho_2) \} \geq \mathcal{P}(\rho)$ and so
$$  \| \mathcal{P}(\rho_1) - \mathcal{P}(\rho_2)\|_{L^1} = 2- 2\int \min\{ \mathcal{P}(\rho_1) , \mathcal{P}(\rho_2) \}\dd x  \leq 2- 2\int \mathcal{P}(\rho)\dd x = 2-2 \int \rho\dd x = \| \rho_1 -\rho_2 \|_{L^1}, $$
which proves the claim.
\end{proof}

\vspace{1cm}

{\sc Acknowledgements.} We would like to thank Filippo Santambrogio for the fruitful discussions during this project and for the careful reading of various versions of the manuscript. The authors warmly acknowledge the support of the ANR project ISOTACE (ANR-12-MONU-0013). We also thank the referee for her/his very useful comments and remarks on the manuscript.


\begin{thebibliography}{99}

\bibitem{AleKimYao} {\sc D. Alexander, I. Kim, Y. Yao}, Quasi-static evolution and congested crowd transport, {\it Nonlinearity}, 27 (2014), No. 4, 823-858. 

\bibitem{AmbGigSav} {\sc L. Ambrosio, N. Gigli, G. Savar\'e}, {\it Gradient flows in metric spaces and in the space of probability measures,} Lectures in Mathematics ETH Z\"urich, Birkh\"auser Verlag, Basel, (2008).

\bibitem{Bou} \textsc{J. E. Bouillet}, Nonuniqueness in $L^{\infty}$: an example, {\it Lecture notes in pure and applied mathematics,} 148 (1993), Differential equations in Banach spaces,   35-40. 


\bibitem{Carrillo} \textsc{J. Carrillo}, Entropy Solution for Nonlinear Degenerate Problems, {\it Arch. Rational Mech. Anal.,} 147 (1999),  269-361. 


 \bibitem{Cha1} \textsc{C. Chalons}, Numerical Approximation of a macroscopic model of pedestrian flows, {\it SIAM J. Sci. Comput.}, Vol. 29 (2007), Issue 2, 539-555.

    \bibitem{Col} \textsc{R.M. Colombo, M.D. Rosini}, Pedestrian Flows and non-classical shocks,  {\it Math. Mod. Meth. Appl. Sci.},  28 (2005), 1553-1567. 
   
    \bibitem{Cos} \textsc{V. Coscia, C. Canavesio}, First-Order macroscopic modelling of human crowd dynamics, {\it Math. Mod. Meth. Appl. Sci.}, 18 (2008), 1217-1247.
      
      
\bibitem{CriPicTos} {\sc E. Cristiani, B. Piccoli, A. Tosin,} \textit{Multiscale Modeling of Pedestrian Dynamics}, Springer, (2014).

\bibitem{Cro} {\sc A.B. Crowley}, On the weak solution of moving boundary problems, {\it J. Inst. Math. Applics}, (1979) 24, 43-57.

\bibitem{DiMoSa} {\sc S. Di Marino, B. Maury,  F. Santambrogio,} Measure sweeping processes, {\it Journal of Convex Analysis}, vol. 23, 2 (2016).

\bibitem{Helb1} \textsc{D. Helbing}, A fluid dynamic model for the movement of pedestrians, {\it Complex Systems}, 6 (1992), 391-415. 
    
\bibitem{Helb3} \textsc{D. Helbing, P. Moln\'ar}, Social force model for pedestrian dynamics, {\it Phys. Rev E} 51 (1995), 4282-4286.
 
\bibitem{Hend} \textsc{L.F. Henderson}, The statistics of crowd fluids, {\it Nature, } 229 (1971), 381-383. 

\bibitem{Hug1} \textsc{R. L. Hughes}, A continuum theory for the flow of pedestrian, {\it Transportation research Part B}, 36 (2002), 507-535. 

\bibitem{Hug2} \textsc{R. L. Hughes}, The flow of human crowds, {\it Annual review of fluid mechanics,} Vol. 35. (2003), Annual Reviews: Palo Alto, CA, 169-182. 

\bibitem{Igb} \textsc{N. Igbida, J. M. Urbano}, Uniqueness for nonlinear degenerate problem, {\it NoDEA Nonlinear Differential Equations Appl.,} 10 (2003),  No. 3, 287-307. 

\bibitem{Igb2} {\sc N. Igbida},  Hele-Shaw type problems with dynamical boundary conditions, {\it J. Math. Anal. Appl.}, 335 (2007), No. 2, 1061-1078.

\bibitem{IgbShiWit} {\sc N. Igbida, K. Sbihi, P. Wittbold}, Renormalized solution for Stefan type problems: existence and uniqueness,
{\it NoDEA Nonlinear Differential Equations Appl.}, 17 (2010), No. 1, 69-93.


\bibitem{krylov} {\sc N.V. Krylov}, {\it Lectures on elliptic and parabolic equations in Sobolev spaces}, Graduate Studies in Mathematics, 96, AMS, (2008).

\bibitem{lady} {\sc O. A. Ladyzenskaja, V. A. Solonnikov, N.N. Uralceva}, {\it Linear and quasilinear equations of parabolic type}, {Translated from the Russian}, Translations of Mathematical Monographs, Vol. 23, (1968).

\bibitem{lasry1} {\sc J.-M. Lasry, P.-L. Lions}, Jeux \`a champ moyen I. Le cas stationnaire, \textit{C. R. Math. Acad. Sci. Paris}, 343 (2006), No. 9, 619-625.

\bibitem{lasry2} {\sc J.-M. Lasry, P.-L. Lions},  Jeux \`a champ moyen II. Horizon fini et contr\^ ole optimal, \textit{C. R. Math. Acad. Sci.
Paris}, 343 (2006), No. 10, 679-684.

\bibitem{lasry3} {\sc J.-M. Lasry, P.-L. Lions,} Mean field games, \textit{Jpn. J. Math.}, 2 (2007), No. 1, 229-260.

\bibitem{lieberman} {\sc G. M. Lieberman}, Second order parabolic differential equation, \textit{World Scientific Publishing}, (1996).


\bibitem{MauRouSan1} {\sc B. Maury, A. Roudneff-Chupin, F. Santambrogio,} A macroscopic crowd motion model of gradient flow type, \textit{Math. Models and Meth. in Appl. Sci.}, 20 (2010), No. 10, 1787-1821.

\bibitem{MauRouSan2} {\sc B. Maury, A. Roudneff-Chupin, F. Santambrogio,} Congestion-driven dendritic growth, 
{\it Discrete Contin. Dyn. Syst.}, 34 (2014), no. 4, 1575-1604.

\bibitem{MauRouSanVen} {\sc B. Maury, A. Roudneff-Chupin, F. Santambrogio, J. Venel, } Handling congestion in crowd motion modeling,  \textit{Netw. Heterog. Media}, 6 (2011), No. 3, 485-519.


\bibitem{crowd1} \textsc{B. Maury, J. Venel}, Handling of contacts in crowd motion simulations, {\it Traffic and Granular Flow, } Springer (2007).

\bibitem{mccann} {\sc R.J. McCann}, A convexity principle for interacting gases, {\it Adv. Math.}, 128 (1997), No. 1, 153-179.

\bibitem{Mes} {\sc A.R. M\'esz\'aros}, Mean Field Games with density constraints, {\it MSc thesis, \'Ecole Polytechnique, Palaiseau, France,} (2012), available at \url{http://www.math.ucla.edu/~alpar/theses/msc_thesis_polytechnique.pdf}. 

\bibitem{MesSan} {\sc A.R. M\'esz\'aros, F. Santambrogio}, Advection-diffusion equations with density constraints, {\it Analysis \& PDE}, { to appear}.

\bibitem{NatPelSav} {\sc L. Natile, M. Peletier, G. Savar\'e}, Contraction of general transportation costs along solutions to Fokker-Planck equations with monotone drifts, {\it J. Math. Pures Appl.}, 95 (2011), 18-35.

\bibitem{Otto} {\sc F. Otto}, $L^1$-contraction and uniqueness for quasilinear elliptic-parabolic equations, {\it J. of Diff. Eq.}, 131 (1996), 20-38.  

\bibitem{PetQuiVaz} {\sc B. Perthame, F. Quir\'os, J.L. V\'azquez}, The Hele-Shaw asymptotics for mechanical models of tumor growth, {\it Arch. Rational Mech. Anal.}, 212 (2014), 93-127.  

\bibitem{Por} {\sc A. Porretta}, Weak Solutions to Fokker-Planck Equations and Mean Field Games, {\it Arch. Rational Mech. Anal.}, 216 (2015), Issue 1, 1-62.

\bibitem{aude:phd} {\sc A. Roudneff-Chupin,} Mod\'elisation macroscopique de mouvements de foule, PhD Thesis, Universit\'e Paris-Sud, (2011), available at \url{http://www.math.u-psud.fr/~roudneff/Images/these_roudneff.pdf}.

\bibitem{OTAM} {\sc F. Santambrogio},  {\it Optimal Transport for Applied Mathematicians,}  Progress in Nonlinear Differential Equations and Their Applications 87, Birkh\"{a}user, Basel (2015).

\bibitem{villani} {\sc C. Villani} {\it Topics in Optimal Transportation}. Graduate Studies in Mathematics, AMS, (2003).

\end{thebibliography}
\end{document}